\documentclass[12pt,a4paper]{article}
\usepackage[left=30mm,top=27mm,right=30mm,bottom=25mm]{geometry}
\usepackage{indentfirst}
\usepackage{amsmath, amsthm, amssymb, times, amsfonts}
\usepackage{enumerate}
\usepackage{eufrak}
\usepackage[mathcal]{eucal}
\usepackage{latexsym}
\usepackage[pdftex]{graphicx}

\begin{document}

\begin{center}
	\textsc{Graph limits, groups and stochastic processes}\\
	\textsc{Workshop}\\
	\textsc{29 June -- 2 July 2014, MTA R\'enyi Institute and E\"otv\"os Lor\'and University, Budapest}
	\vspace{1cm}
	\hrule
	\vspace{1cm}
	{\Large\textsc{Measurable equidecompositions via combinatorics and group theory
}}\\[1cm] % !!!PLEASE INSERT THE RELEVANT PARAMETERS HERE
	{\Large {\L}ukasz Grabowski\footnote{Supported by EPSRC grant~EP/K012045/1}, Andr\'as M\'ath\'e\footnote{Supported by a Levehulme Trust Early Career Fellowship}, Oleg Pikhurko\footnote{Supported by ERC
grant~306493 and EPSRC grant~EP/K012045/1}}\\[3mm]{\large University of Warwick}
\end{center}

\vspace{0.5cm}

\begin{center}
	{\large \bf Abstract}
\end{center}

%PLEASE INSERT THE ABSTRACT HERE
\newcommand{\R}{{\mathbb R}}
\newcommand{\E}{\mathbf{E}}

\newcommand{\I}[1]{{\mathbb #1}}
\newcommand{\sphere}{{\mathbb S}}
\newcommand{\ga}{\gamma}
\newtheorem{theorem}{Theorem}

We give a sketch of proof that any two (Lebesgue) measurable subsets of the unit sphere in $\R^n$, for $n\ge 3$, with non-empty interiors and of the same measure are equidecomposable using pieces that are measurable. 
\vspace{0.5cm}
\hrule
\vspace{0.5cm}
Recall that two subsets $A$ and $B$ of $\R^n$ are  \textit{equidecomposable} if  for some $m\,{\in}\,\mathbb N$ there exist isometries $\gamma_1,\dots,\gamma_m$ and partitions $A=A_1\sqcup\dots\sqcup A_m$ and $B=B_1\sqcup\dots\sqcup B_m$ such that $\gamma_i(A_i)=B_i$ for each $i\in\{1,\dots,m\}$. Equidecomposability for subsets of the unit sphere $\mathbb S^{n-1}\subseteq \R^n$ is defined likewise, with $\ga_i$'s being rotations.

Banach and Tarski \cite{banach+tarski:24} proved that, for $n\ge 3$, any two bounded subsets of $\R^n$ with non-empty interiors  are equidecomposable. 
Using earlier results of Banach~\cite{Banach1923}, they also showed that the above statement is false in $\R^2$ and $\R^1$. Indeed, in these cases the Lebesgue measure can be extended to an isometry-invariant finitely additive measure defined on all subsets, and so equidecomposable sets which are measurable must necessarily have the same measure. 

In this context, Tarski \cite{tarski:25} posed the following question: Is a $2$-dimensional disk equidecomposable with a $2$-dimensional square of the same area? This problem became known as \textit{Tarski's circle squaring}.  Some 65 years later, Laczkovich \cite{laczkovich:90} showed  that Tarski's circle squaring is indeed possible.

A natural further question is the existence of a \emph{measurable equidecomposition} where each piece $A_i$ has to be (Lebesgue) measurable.  In opposition to the Banach--Hausdorff--Tarski-type paradoxes when pieces are rearranged to form another set of \textit{different} measure, there is no obvious reason for Tarski's circle squaring to be impossible with measurable or even Borel pieces (\cite[Appendix C, Question 2]{wagon:btp}).

Although we do not resolve the possibility of measurable Tarski's circle squaring on the plane, we provide the answer for the analogous question in three and more dimensions. To simplify the discussion we do not present the most general statements; they will be included in the full version of this article (in preparation).

\begin{theorem}\label{thm-euclidean}
Let $n\ge 3$ and let $A,B\subseteq\R^n$ be two bounded measurable sets with non-empty interiors and of the same measure. Then $A$ and $B$ are measurably equidecomposable.
\end{theorem}

Apart from the higher dimensional measurable Tarski's circle squaring, Theorem~\ref{thm-euclidean} also answers in the affirmative Question 3.14 from Wagon's book \cite{wagon:btp} whether a 
regular tetrahedron and a cube in $\R^3$ of the same volume are measurably equidecomposable. This can be viewed as the measurable version of Hilbert's third problem. 

To indicate the methods used, we give a sketch of the proof of the analogous theorem for subsets of $\sphere^{n-1}$, and only ``up to measure zero". Afterwards we briefly mention what is needed to pass to Euclidean spaces and obtain exact equidecompositions (i.e.\ without removing subsets of measure zero).

\begin{theorem}\label{thm-spheres}
Let $n\ge 3$ and let  $A,B\subseteq\sphere^{n-1}$ be (Lebesgue) measurable sets with non-empty interiors and of the same measure. Then there exist sets $A',B'\subseteq \sphere^{n-1}$ of measure zero such that $A\setminus A'$ and $B\setminus B'$ are equidecomposable with pieces that are Borel.  
\end{theorem}
\begin{proof}[Sketch of proof] To simplify the notation assume that $A$ and $B$ are disjoint.

Let $\mu$ be the uniform measure on $\sphere^{n-1}$ normalised to be a probability measure. A key ingredient is the following \textit{spectral gap property}. There exist rotations 
$\ga_1,\ldots, \ga_k\in SO(n)$ and a real $c>0$ such that the averaging operator $T\colon L^2(\sphere^{n-1},\mu) \to L^2(
\sphere^{n-1},\mu)$ defined by
 $$
 (Tf)(x)=\frac1{k} \sum_{i=1}^k f(\ga_i(x)),\qquad f\in L^2(\sphere^{n-1},\mu),\ x\in\sphere^{n-1},
 $$
 satisfies $\|Tf\|_2\le (1-c)\,\|f\|_2$ for every $f\in L^2(\sphere^{n-1},\mu)$ with $\int f(x)\,d\mu(x)=0$.

The existence of such rotations was shown independently by Margulis~\cite{margulis:80} and Sullivan~\cite{sullivan:81} for $n\ge 5$, and then by Drinfel'd~\cite{drinfeld:84} for the remaining cases $n=3,4$.
(Note that we cannot have spectral gap for $n=2$ as the group $SO(2)$ is Abelian.)

Using \cite[Proposition 6.3.2 and Remark 6.3.3(ii)]{bekka+harpe+valette:kpt}), it can be shown that the spectral gap is actually equivalent to the following \emph{expansion property}. For every $\eta>0$ there is a finite set $S$ of rotations such that for every measurable subset $U\subseteq \sphere^{n-1}$ we have
\begin{equation}\label{eq:admits-expansion}
\mu \left(\,\cup_{\ga\in S}\, \ga(U)\,\right) \ge \min(1-\eta,\,\mu(U)/\eta).
\end{equation}

Since $A$ and $B$ have non-empty interiors and $\sphere^{n-1}$ is compact, we can find a finite set $T$ of rotations such that $\sphere^{n-1} = \cup_{\ga\in T}\, \ga(A) = \cup_{\ga\in T}\, \ga(B)$. Let us fix $\eta>0$ such that $\eta < \min(\mu(A)/3,(2\,|T|)^{-1})$, and let $S$ be as in \eqref{eq:admits-expansion}. 
By enlarging $S$, we can assume that $S$ is \emph{symmetric}, that is, $S^{-1}=S$. Let 
 $$
 R:=\{\tau^{-1}\gamma\colon \gamma\in S,\ \tau\in T\}\cup \{\gamma\tau \colon \gamma\in S,\ \tau\in T\}\subseteq SO(n).
 $$
 Note that $R$ is also symmetric.

Consider the bipartite graph $\cal G$ whose set of vertices is $A\cup B$ and with an edge between $x\in A$ and $y\in B$ if for some $\ga\in R$ we have $\ga(x)=y$.

\vspace{8pt}\noindent\textbf{Claim 1.} Let $U$ be a measurable set contained entirely either in $A$ or in $B$, and let $N(U)$ be the set of neighbours of $U$ in $\cal G$. Then
\begin{equation}\label{eq-hall-condition}
\mu(N(U)) \ge  \min\left(\frac{2}{3}\,\mu(A), 2\mu(U)\right).
\end{equation}
\begin{proof}[Proof of Claim]
By symmetry, assume that $U\subseteq A$.  There are
two cases to consider when we apply \eqref{eq:admits-expansion} to $U$.

Suppose first that $\mu(\,\cup_{\ga\in S}\, \ga(U)\,)\ge 1-\eta$. The set $X:=\cup_{\ga\in S}\, \ga(U)$ satisfies $\mu(\sphere^{n-1}\setminus X)\le \eta<\mu(A)/3$. In particular, fixing
some $\tau\in T$, we have that
$N(U)\supseteq \tau^{-1}(X)\cap B$ cover at least $2/3$ of the measure of $B$, as required.

It remains to consider the case when  the minimum  in \eqref{eq:admits-expansion}  
is given by $\mu(U)/\eta$. By the choice of $\eta$ we have
$$
\mu \left(\,\cup_{\ga\in S}\, \ga(U)\,\right) \ge \mu(U)/\eta \ge 2\,|T|\,\mu(U).
$$
Since the whole sphere is covered by $|T|$ copies of $B$, there is $\tau\in T$ such that
$$
\mu \left(\tau(B)\cap \cup_{\ga\in S}\, \ga(U)\right) \ge 2\,\mu(U).
$$
By definition $\tau^{-1}\gamma\in R$ for every $\gamma\in S$. Therefore, we get that
$$
\mu(N(U)) = \mu \left(B\cap \cup_{\ga\in R}\, \ga(U)\right) \ge \mu \left(\tau(B)\cap \cup_{\ga\in S} \,\ga(U)\right) \ge 2\,\mu(U),
$$
 proving Claim~1.\mbox{}\end{proof} 

Recall that a \textit{matching} in a graph is a set of edges such that no two edges  share a vertex. 
A matching $M$ in the graph $\cal G$ is called \textit{Borel} if there exist disjoint Borel subsets $A_\ga\subseteq A$ indexed by $\ga\in R$ such that 
 \begin{equation}\label{eq:M}
 M=\cup_{\gamma\in R} \big\{\{x,\gamma(x)\} : x\in A_\gamma\big\}.
\end{equation}

Clearly, in order to finish the proof it is enough to find a Borel matching in $\cal G$ such that the set of unmatched vertices has measure zero. As noted by Lyons and Nazarov in \cite[Remark 2.6]{lyons+nazarov:11}, the expansion property \eqref{eq-hall-condition} suffices for this. 
Let us outline their strategy. 

Recall that an \textit{augmenting path} for a matching $M$ is a path which starts and ends at an unmatched vertex and such that every second edge belongs to $M$. A \textit{Borel augmenting family} is a Borel subset $U\subseteq A\cup B$ and a finite sequence $\ga_1,\ldots, \ga_l$ of elements of $R$ such that (i) for every $x\in U$ the sequence $y_0,\dots,y_l$, where
$y_0=x$ and $y_j=\gamma_j(y_{j-1})$ for $j=1,\dots,l$,
forms an augmenting path and (ii) for every distinct $x,y\in U$ the corresponding augmenting paths are vertex-disjoint.

As shown by Elek and Lippner \cite{elek+lippner:10}, there exists a sequence $(M_i)_{i\in\I N}$ of Borel matchings such that $M_i$ admits no augmenting path of length at most $2i-1$ and $M_{i+1}$ can be obtained from $M_i$ by iterating the following at most countably many times: pick some
Borel augmenting family $(U,\gamma_1,\dots,\gamma_l)$ with $l\le 2i+1$ and flip (i.e.\ augment) the current matching along all paths given by the family. See \cite{elek+lippner:10} for more details.

Our task now is to show, using Claim~1, that the measure of of vertices not matched by $M_i$ tends to zero as $i\to\infty$ and that the sequence $(M_i)_{i\in \I N}$ stabilises a.e.

Let us fix $i\ge 1$ and let $X_0$ be the subset of $A$ consisting of vertices that are not matched by $M_i$. An \textit{alternating path of length $l$} is a sequence of distinct vertices $x_0,\dots,x_l$ such that (i) $x_0\in X_0$, (ii) for odd $j$ we have $x_jx_{j+1}\in M_i$, and
(iii) for even $j$ we have $x_jx_{j+1}\in E({\cal G})\setminus M_i$. Let $X_j$ consist of the end-vertices of alternating paths of length at most $j$.
Clearly for all $j$ we have $X_j\subseteq X_{j+1}$ and so, in particular, $\mu(X_{j+1})\ge \mu(X_j)$. 
For $j\ge 1$, let $X_j':= X_j\setminus X_{j-1}$.

\vspace{8pt}\noindent\textbf{Claim 2.} For every odd $j\le 2i-1$ we have $\mu(X_{j}') = \mu(X_{j+1}')$ and $\mu(X_j\cap B)\le \mu(X_{j+1}\cap A)$.

\begin{proof}[Proof of Claim]
All vertices in $X_j'$ are covered by the matching $M_i$, for otherwise we would have an augmenting path of length $j$.
It follows that $M_i$ gives a bijection between $X_{j}'$ and $X_{j+1}'$. If we take the sets $A_\gamma$
that represent $M_i$ as in (\ref{eq:M}), then the partitions $\cup_{\gamma\in R}\, A_\gamma$ and $\cup_{\gamma\in R}\, \gamma(A_\gamma)$ induce a Borel equidecomposition between $X_{j}'$ and $X_{j+1}'$, so these sets
have the same measure, as required. 

The second part (i.e.\ the inequality) follows analogously from the fact that $M_i$ gives an 
injection of $X_j\cap B$ into $X_{j+1}\cap A$ (with $X_0$ being the set of vertices missed by this injection).
\end{proof}

Let $k$ be even, with $2\le k\le 2i-2$. Let $U=X_k\cap A$. We have that 
$N(U)= X_{k+1}\cap B$. By Claim~1, 
$$
\mu(X_{k+1}\cap B)=\mu(N(U))\ge \min\left(\frac23\,\mu(A),
2\mu(U)\right). $$  
 
If $\mu(X_{k+1}\cap B)\ge \frac23\,\mu(A)$ then, by Claim~2, $\mu(X_{k+2}\cap A)\ge \mu(X_{k+1}\cap B)\ge
\frac23\,\mu(A)$ and thus 
 $$
 \mu(X_{k+2})= \mu(X_{k+1}\cap B)+\mu(X_{k+2}\cap A)\ge \frac43\, \mu(A).
 $$ 

Now, suppose that $\mu(X_{k+1}\cap B)\ge 2\mu(U)$. By applying Claim~2 for $j=k-1$ we obtain
 $$
 \mu(X_{k+1}')=\mu(X_{k+1}\cap B)-\mu(X_{k-1}\cap B)\ge 2\mu (U)-\mu (U)=\mu (U).
 $$
 Again, by Claim~2, $\mu(X_{k+2}')=\mu(X_{k+1}')$ and 
$\mu(X_k)=\mu(X_{k-1}\cap B)+\mu(U)\le 2\mu(U)$. Thus 
 $$
 \mu(X_{k+2})=\mu(X_k)+\mu(X_{k+1}')+\mu(X_{k+2}')\ge \mu(X_k)+2\mu(U)\ge 2 \mu(X_k).
 $$ 

 Thus the measure of $X_k$ expands by factor at least $2$
when we increase $k$ by 2, unless $\mu(X_{k+2})\ge \frac43\,\mu(A)$. Also, this conclusion formally holds
for $k=0$, when $X_1=N(X_0)$. 

By using induction, we  conclude that, for all even $k$ with $0\le k\le 2i$,
 \begin{equation}\label{eq:Xk}
	\mu(X_k) \ge \min\left(\frac43\, \mu(A), 2^{k/2} \mu(X_0)\right).
 \end{equation}

In the same fashion we define $Y_0$ to be the subset of $B$ consisting of vertices not matched by $M_i$ and let  $Y_j$ consist of the end-vertices of alternating paths that start in $Y_0$ and have length at most $j$. As before, we obtain that $Y_j$'s satisfy the analogue of~(\ref{eq:Xk}).

The sets $X_{i-1}$ and $Y_{i}$ are disjoint for otherwise we would find and augmenting path of length at most $2i-1$. It follows that they cannot each have measure more than $\mu(A) =\frac12\, \mu(A\cup B)$. Since $\mu(X_0) = \mu(Y_0)$ we conclude that
 \begin{equation}\label{eq:X0Y0}
 \mu(X_0\cup Y_0) \le 2\,\mu(A)\cdot \left(\frac{1}{2}\right)^{\lfloor(i-1)/2\rfloor}.
 \end{equation}

As we noted before, $M_{i+1}$ arises from $M_i$ by flipping 
augmenting paths of length at most $2i+1$ in a Borel way. When one such path is flipped, two
vertices are removed from the current set $X_0\cup Y_0$ of unmatched vertices.  Using this observation and the fact that each rotation is measure-preserving, one can show that the set of vertices covered by the symmetric difference $M_{i+1}\triangle M_i$ has measure at most $(2i+2)\mu(X_0\cup Y_0)/2$. We know by 
(\ref{eq:X0Y0}) that this goes to 0 exponentially fast with $i$; in particular, it is summable over $i\in\I N$. The Borel-Cantelli Lemma implies that the sequence of matchings $(M_i)_{i\in\I N}$ stabilises a.e. 

This finishes the sketch of proof of Theorem~\ref{thm-spheres}. 
\mbox{}\end{proof}

The argument that leads to (\ref{eq-hall-condition}) can be adopted to show that $|N(X)|\ge 2\, |X|$ for every \textit{finite} subset $X$ of $A$ (and of $B$). In the slightly simpler 
case when both $A$ and $B$ are open 
subsets of $\sphere^{n-1}$, the desired bound on $|N(X)|$ can be deduced by applying Claim~1 to the union of spherical caps of sufficiently small radius $\epsilon>0$ around points of $X$. By a result of Rado~\cite{rado:42},
this guarantees that each connectivity component of $\cal G$ has a perfect matching. 
The (exact) measurable equidecomposition of $A$ and $B$ can be obtained by modifying the Borel equidecomposition from Theorem
\ref{thm-spheres} on suitably chosen sets of measure zero where we use Rado's theorem and the Axiom of Choice. Since all sets of measure zero are measurable, we obtain a measurable decomposition between
$A$ and $B$ (without any exceptional sets $A'$ and $B'$).

The spectral gap property for $\R^n$, as stated in the proof of Theorem \ref{thm-spheres}, fails. However, we could argue that a suitable reformulation of the expansion property \eqref{eq:admits-expansion} still holds for bounded sets (while the rest of the proof of Theorem~\ref{thm-euclidean} is essentially the same as above).

\bibliographystyle{plain}
%\bibliography{measurable-dec,sets}

\end{document}